\documentclass{psp}
\usepackage{amsmath,amssymb,amsfonts,enumitem}

\newtheorem{prop}{Proposition}
\newtheorem{lem}[prop]{Lemma}
\newtheorem{theorem}[prop]{Theorem}
\newtheorem{cor}[prop]{Corollary}

\newtheorem{rmk}[prop]{Remark}

\newcommand{\Zb}{\mathbf{Z}}

\newcommand{\Qb}{\mathbf{Q}}

\newcommand{\XX}{\mathfrak X}

\newcommand{\mc}[1]{\mathcal{#1}}

\newcommand{\mf}[1]{\mathfrak{#1}}

\newcommand{\triv}{\{1\}}
\newcommand{\inv}{^{-1}}

\newcommand{\Aut}{\mathrm{Aut}}

\newcommand{\normal}{\trianglelefteq}

\newcommand{\proj}{\mathrm{pr}}

\newcommand{\FC}{\mathrm{FC}}

\newcommand{\wt}[1]{\widetilde{#1}}

\usepackage[usenames, dvipsnames]{color}

\begin{document}
\title[Residual and profinite closures of commensurated subgroups]{On the residual and profinite closures\\ of  commensurated subgroups}
\author[\small{P.-E. Caprace, P. Kropholler, C. Reid and P. Wesolek}]{PIERRE-EMMANUEL CAPRACE \\
Universit\'e Catholique de Louvain, IRMP,
\addressbreak
Chemin du Cyclotron 2, bte L7.01.02, 1348 Louvain-la-Neuve, Belgique \addressbreak 
e-mail\textup{: \texttt{pe.caprace@uclouvain.be}}
\thanks{F.R.S.-FNRS senior research associate, supported in part by EPSRC grant no EP/K032208/1.}
\nextauthor
PETER H. KROPHOLLER \\
Mathematical Sciences, 
University of Southampton,  UK
\addressbreak 
e-mail\textup{: \texttt{p.h.kropholler@soton.ac.uk}}
\thanks{supported by  EPSRC grants no EP/K032208/1 and EP/ N007328/1.}
\nextauthor
COLIN D. REID \\
University of Newcastle, \addressbreak
School of Mathematical and Physical Sciences, \addressbreak
Callaghan, NSW 2308, 
Australia \addressbreak
e-mail\textup{: colin.d.reid@newcastle.edu.au}
\thanks{ARC DECRA fellow, supported in part by ARC Discovery Project DP120100996.}
\nextauthor
PHILLIP WESOLEK\\
Binghamton University, \addressbreak
Department of Mathematical Sciences, PO Box 6000, 	\addressbreak
Binghamton, New York 13902-6000, USA \addressbreak
e-mail\textup{: pwesolek@binghamton.edu}
}

\volume{???}
\pubyear{????}
\setcounter{page}{1}
\receivedline{Received \textup{10} July \textup{2017;}
              revised \textup{4} June \textup{2019}}

\maketitle

\begin{abstract} 
The residual closure of a subgroup $H$ of a group $G$ is the intersection of all virtually normal subgroups of $G$ containing $H$. We show that if $G$ is generated by finitely many cosets of $H$ and if $H$ is commensurated, then    the residual closure   of $H$ in $G$ is  virtually normal. This implies that  separable commensurated subgroups of finitely generated groups are virtually normal. A stream of applications   to separable subgroups, polycyclic groups, residually finite groups, groups acting on trees, lattices in products of trees and just-infinite groups then flows from this main result. 
\end{abstract}

\tableofcontents

This paper consists of a Main Theorem concerning commensurated subgroups and topologies related to profinite topologies. It then has a string of Corollaries. Some of the Corollaries are original, some known, some are well known. For example the applications to residually finite groups in Section 4 and some of the Corollaries in Section 8 are to the best of our knowledge new. But the point is to show that all our results flow in a natural way from a single source. Some of the Corollaries deserve to be called, or are called, Theorems in their own right. Before coming to the formal and complete statement of the Main Theorem we examine a special case belonging to very familiar territory.
Let $G$ be a finitely generated group and let $H$ be a subgroup that satisfies the following two conditions:
\begin{enumerate}
\item[$\mathbf{A.}$]
$H$ is separable in $G$, meaning that it is an intersection of subgroups of finite index or equivalently closed in the profinite topology.
\item[$\mathbf{B.}$]
$H$ is commensurate with all of its conjugates, meaning that $|H:H\cap gHg^{-1}|$ is finite for all $g\in G$.
\end{enumerate}
Then we prove that there is a subgroup $K$ of $H$ such that
$K$ is normal in $G$ and of finite index in $H$.

This very simple general observation is not difficult to prove. Our original proofs used moderately sophisticated tools from the theory of totally disconnected locally compact groups, but here we present a proof that could be delivered in any beginning graduate course in abstract group theory. Startlingly this leads to simplifications and streamlinings of many results scattered through the literature, as we shall show. 

We say that a subgroup $H$ of a group $G$ is commensurated when it is commensurate with all its conjugates. The commensurator of $H$ is the unique largest subgroup of $G$ in which $H$ is a commensurated subgroup.

For a first example of an application, if $G$ is a polycyclic group then all its subgroups satisfy $\mathbf{A.}$ So any one that satisfies $\mathbf {B}$ is commensurate to a normal subgroup. More generally, our main theorem shows that every subgroup of a polycyclic group is commensurate with a subgroup that is normal in its commensurator, see the implication Theorem 10 (i)$\implies$(iv) below: a significant streamlining of a fundamental lemma of Kropholler \cite{Kropholler}. But the larger picture which involves a theorem of Jeanes and Wilson can also be given a smooth treatment using our Main Theorem.

Further applications can be made to group actions on graphs, metric spaces, and trees, to Baumslag--Solitar groups, the Grigorchuk group and at least one member of the family of Gupta--Sidki groups and more.

Moreover it is rather quickly clear that generalizations are possible: for example if condition $\mathbf{A}$ is not satisfied then one can still apply the result to the profinite closure of the subgroup.
Upon arriving at preliminary forms our Main Theorem it became apparent that profinite closure and separability, while very familar concepts, should be treated alongside the slightly more subtle and less well known concepts of \emph{residual closure} and \emph{weak separability}. This generality created no additional mathematical difficulty. So without further ado, let us proceed to the general form of

\section{The Main Theorem.}

Let $G$ be a group. A subgroup $H$ is called \textbf{virtually normal} if there is a subgroup $K\le H$ that is normal in $G$ and of finite index in $H$. A subgroup is called \textbf{weakly separable} if it is an intersection of virtually normal subgroups. Any intersection of  weakly separable subgroups is weakly separable. Any subgroup $H \leq G$ is thus contained in a unique smallest weakly separable subgroup, denoted by $\widetilde H$ and  called the  \textbf{residual closure} of $H$ in $G$. 

Weakly separable subgroups are a generalization of separable subgroups. A \textbf{separable subgroup} of $G$ is an intersection of subgroups that have finite index in $G$. The \textbf{profinite closure} of a subgroup $H$ in $G$, denoted by $\overline{H}$, is the unique smallest \textit{separable} subgroup containing $H$. The inclusion  $\wt{H} \le \overline{H}$ holds in general, but it  can be strict. For instance, if $G$ is infinite and simple and $H$ is any  finite subgroup,   then $\wt{H} = H$, but $\overline{H} = G$.  

  
\begin{thm*}\label{thm:ProfiniteDensity}
Let $G$ be a   group and $H \leq G$ be a commensurated subgroup. Assume that $G$ is generated by finitely many cosets of $H$. Then 
$$N  := \bigcap_{g \in G} g \widetilde H g^{-1}$$
is of finite index in the residual closure $\widetilde H$ of $H$.  Additionally, $N$ is a normal subgroup of $G$ such that $[H: N \cap H]< \infty$, $N = \widetilde{N \cap H}$ and $\widetilde{H} = NH$.
\end{thm*}

Basically, in words and in the case when the ambient group is finitely generated, it says that the residual closure of a commensurated subgroup is virtually normal.

For the proof, we need a couple of lemmas to show that the operation of taking the residual closure is well behaved. Define $\mathcal{N}(G,H)$ to be the set of normal subgroups of $G$ that contain a finite index subgroup of $H$. In particular  $\mathcal{N}(G,G)$ is the set of normal subgroups of finite index in $G$. The following basic fact will be used repeatedly.

\begin{lem}\label{lem:N(G,H)}
Let $G$ be a group and $H \leq G$ be a subgroup. Then 
$$\widetilde{H} = \bigcap_{N \in \mathcal{N}(G,H)}NH, 
\hspace{1cm}\text{and} \hspace{1cm}
\overline{H} = \bigcap_{N \in \mathcal{N}(G,G)}NH.$$
\end{lem}
\begin{proof}
For any $N \in \mathcal{N}(G,H)$, the group $NH$ is a virtually normal subgroup of $G$ containing $H$. Thus 	$\widetilde{H} \leq \bigcap_{N \in \mathcal{N}(G,H)}NH $. 
Conversely, let  $J$ be any virtually normal subgroup of $G$ containing $H$. Let $N$ be a finite index subgroup of $J$ which is normal in $G$. We then have $NH \leq J$ and  $N \in  \mathcal{N}(G,H)$. Hence  	$\widetilde{H} \geq \bigcap_{N \in \mathcal{N}(G,H)}NH $. 

The proof for the profinite closure is similar.
\end{proof}

Two subgroups $H_1, H_2 \leq G$ are called \textbf{commensurate} if their intersection is of finite index in both $H_1$ and $H_2$.

\begin{lem}\label{lem:closure_commensurability}
Let $G$ be a group and let $K \le H \le G$ such that $[H:K]<\infty$.  Then
\[
\mathcal{N}(G,H) = \mathcal{N}(G,K), \;  \widetilde{H} = \widetilde{K}H, \text{ and } \;\overline{H}=\overline{K}H
\]
In particular, if $H_1$ and $H_2$ are commensurate, then $\widetilde{H_1}$ and $\widetilde{H_2}$ are commensurate and  $\overline{H_1}$ and $\overline{H_2}$ are commensurate. If $H$ is commensurated in $G$, then $\widetilde{H}$ is commensurated and $\overline{H}$ is commensurated.
\end{lem}

\begin{proof}
For any subgroup $N\leq G$, $[K:K \cap N] < \infty$ if and only if $[H :H \cap N] < \infty$. We infer that $\mathcal{N}(G,H) = \mathcal{N}(G,K)$.  

For the remaining claims, the cases of the residual closure and the profinite closure are similar. We therefore only consider the residual closure.  Clearly,
\[
K \leq \widetilde K \cap H \leq \widetilde K,
\]
so $\widetilde K = \widetilde{\widetilde K \cap H}$.  In order to prove that $\widetilde{H} = \widetilde{K}H$,  we may thus assume that $K = \widetilde{K} \cap H$.  

Write $H$ as a disjoint union of right cosets $Kt_1,\dots,Kt_n$ of $K$. For each $1 \le i < j \le n$, there is some $N \in \mathcal{N}(G,H)$ such that $t_it\inv_j \not\in NK$ via Lemma~\ref{lem:N(G,H)}, since $K = \widetilde{K} \cap H$ and $\mathcal{N}(G,H) = \mathcal{N}(G,K)$.  The set $\mathcal{N}(G,H)$ is closed under finite intersections, so there is some $M \in \mathcal{N}(G,H)$ such that all of the cosets $MKt_1,\dots,MKt_n$ are distinct. Fixing such an $M$, Lemma~\ref{lem:N(G,H)} ensures that we can write $\widetilde{H}$ as
\[
\bigcap_{N \in \mc{M}}NH
\]
where $\mc{M}$ is the set of $N \in \mathcal{N}(G,H)$ such that $N \le M$.

Take $x \in \widetilde{H}$. We have $x \in MH$, so $x \in MKt_i$ for exactly one $i$.  For each $N \in \mc{M}$, it is also the case that $x \in NH$, so $x \in NKt_i$.  Therefore,
\[
x \in \left(\bigcap_{N \in \mc{M}}NK\right)t_i,
\]
and the latter set is exactly the coset $\widetilde{K}t_i$ by Lemma~\ref{lem:N(G,H)}.  We conclude that $x \in \widetilde{K}H$. Thus, $\widetilde{H}\subseteq \widetilde{K}H$, and as $\widetilde{K}H\subseteq \widetilde{H}$, the proof is complete.
\end{proof}

\begin{lem}\label{lem:closure_normal}
Let $G$ be a group and let $H \le G$.
\begin{enumerate}[label=(\roman{*})]
\item  For any  $N \in \mc{N}(G,H)$, we have $N \cap \widetilde{H} = \widetilde{N \cap H}$ and $\mc{N}(G,H) = \mc{N}(G,\widetilde{H})$. 
\item For any $N\in \mc{N}(G,G)$, we have $N \cap \overline{H} = \overline{N \cap H}$. 
\end{enumerate}
\end{lem}

\begin{proof}
We prove (i). Claim (ii) is similar.

Take  $N \in \mc{N}(G,H)$. By Lemma~\ref{lem:N(G,H)}, we can write $\widetilde{H}$ as
\[
\widetilde{H} = \bigcap_{M \in \mc{M}}MH
\]
where $\mc{M}$ is the set of elements of $\mc{N}(G,H)$ contained in $N$.  Then
\[
\widetilde{H} \cap N = \bigcap_{M \in \mc{M}}(MH \cap N) = \bigcap_{M \in \mc{M}}M(H \cap N) = \widetilde{H \cap N}
\]
where the last equality follows Lemmas~\ref{lem:N(G,H)} and \ref{lem:closure_commensurability}. Applying Lemma~\ref{lem:closure_commensurability} again, $\widetilde{H \cap N}$ has finite index in $\widetilde{H}$, so $N \in \mc{N}(G, \widetilde H)$.  We conclude that $\mc{N}(G,H) \subseteq \mc{N}(G,\widetilde{H})$. The reverse inclusion  is  clear.
\end{proof}

\begin{proof}[Proof of the Main Theorem]
Set $N: = \bigcap_{g \in G} g \widetilde H g^{-1}$. Suppose for the moment that $N$ has finite index in $\widetilde{H}$. We infer that $N \in \mc{N}(G,H)$, so $\widetilde{H} \le NH$. Since $N,H \le \widetilde{H}$, we indeed have that $\widetilde{H} = NH$. Lemma~\ref{lem:closure_normal} implies additionally that $N = \widetilde{N \cap H}$. To prove the theorem, it thus suffices to show that $N$ has finite index in $\widetilde{H}$. By Lemma~\ref{lem:closure_commensurability}, $\widetilde{H}$ is a commensurated subgroup of $G$, so we may assume that $H = \widetilde{H}$.

For $S$ a subset of $G$, we define
\[
H_S:=\bigcap_{s\in S}sHs^{-1}.
\]
If $1\in S$  and $S$ is finite, then $H_S$ has finite index in $H$. The group $H_S$ is weakly separable, since $H$ is weakly separable. Letting $H_St_1,\dots,H_St_n$ list right coset representatives of $H_S$ in $H$, there is some $V\in \mathcal{N}(G,H_S)$ such that $VH_S$ does not contain $t_it_j^{-1}$ for any $i\neq j$, since $\bigcap_{N\in \mathcal{N}(G,H_S)}NH_S=H_S$ by Lemma~\ref{lem:N(G,H)}. It now follows that $H\cap VH_S=H_S$, hence $(H\cap V)H_S=H_S$. We thus deduce that  $ H\cap V \leq H_S$ with $V\in \mathcal{N}(G,H)$; that $V\in \mathcal{N}(G,H)$ is given by Lemma~\ref{lem:closure_commensurability}. We now obtain the following.
 
\begin{enumerate}
	\item[(A)] For all finite $S\subset G$ with $1\in S$, there is $V \in \mathcal{N}(G,H)$ such that for all $T$ with $1 \in T$ and $T\subseteq S$, we have
	$$(H\cap V)H_T = H\cap VH_T=H_T.$$
\end{enumerate}

Fix $V \in \mathcal{N}(G,H)$. Let $\XX(V)$ denote the collection of all finite subsets $S$ of $G$ containing $1$ such that  
$$ H\cap VH_S=H_S.$$ 
Suppose that $S$ and $T$ are subsets of $G$ and $x$ is an element of $G$. We shall prove the following claim.
\begin{enumerate}
	\item[(B)] If $S$, $T$ and $\{1,x\}$ all belong to $\XX(V)$ then $S\cup xT$ belongs to $\XX(V)$.
\end{enumerate}
We have
$$H\cap VH_S=H_S,\eqno(1)$$
$$H\cap VH_T=H_T, \text{ and }\eqno(2)$$
$$H\cap V(H\cap xHx^{-1})=H\cap xHx^{-1}.\eqno(3)$$
Conjugating (2) by $x$ yields
$$xHx^{-1}\cap VH_{xT}=H_{xT}\eqno(4)$$
Intersecting with (1), we obtain
$$H\cap xHx^{-1}\cap VH_S\cap VH_{xT} =H_S\cap H_{xT}=H_{S\cup xT}\eqno(5)$$
Using (3) to substitute for $H\cap xHx^{-1}$ on the left hand side of (5), we obtain
$$H\cap V(H\cap xHx^{-1})\cap VH_S\cap VH_{xT}=H_{S\cup xT}.\eqno(6)$$
\textit{A fortiori}, noting that $1\in T$ and $1\in S$, we have 
$$H_{S\cup xT}\subseteq V(H\cap xHx^{-1})\cap VH_S\cap VH_{xT},$$ and therefore,
$$H\cap VH_{S\cup xT}\subseteq H\cap V(H\cap xHx^{-1})\cap VH_S\cap VH_{xT}.$$
Combining this with (6), we obtain claim (B).

\medskip
Let $Y$ be a finite subset of $G$ such that $1\in Y=Y^{-1}$ and $G=\langle YH\rangle$. Since $H$ is commensurated, every right $H$-coset is contained in a finite union of left $H$-cosets. There is thus a finite subset $X$ of $G$ such that $1\in X=X^{-1}$ and that $HYH = XH$. In particular, we have $HXH = XH$. It follows that $(XH)^n = X^n H$ for all $n$. Hence, $G = \langle X \rangle H$. 

Using (A) we can choose $V$ in $\mc{N}(G,H)$ such that $\XX(V)$ contains $X$ as well as $\{1,x\}$ for all $x \in X$.   By iterated application of (B), we see that $X^n: =\{x_1\dots x_n \mid x_i \in X\}$ belongs to $\XX(V)$ for all $n >0$. We deduce that
$$H \cap V \leq \bigcap_{n\geq 1} H\cap VH_{X^n} =  \bigcap_{n\geq 1} H_{X^n} = \bigcap_{g \in \langle X \rangle} gHg^{-1} \leq H.$$
Since $G = \langle X \rangle H$, the group $\bigcap_{g \in \langle X \rangle} gHg^{-1} = \bigcap_{g \in G} gHg^{-1} $ is normal in $G$,   and its   index in $H$ is bounded above by $[H: H \cap V]$. The result follows.
\end{proof}

We immediately recover a result of Caprace and Monod, which was in fact the original inspiration for the Main Theorem.  A locally compact group is called \textbf{residually discrete} if the intersection of all its open normal subgroups is trivial.

\begin{cor}[Caprace--Monod, {\cite[Corollary 4.1]{CM11}}]\label{prop:ResDiscr->SIN}
A compactly generated 
and 
totally disconnected locally compact group is residually discrete if and only if it has a basis of identity neighborhoods consisting of compact open normal subgroups. 
\end{cor}

\begin{proof}
Let $G$ be a compactly generated totally disconnected locally compact group, $O$ be an identity neighborhood in $G$, and  $\mc{N}$ be the set of open normal subgroups of $G$.  By Van Dantzig's theorem, $O$ contains a compact open subgroup $H$. That $H$ is compact ensures that $[H:H \cap N]< \infty$ for all $N \in \mc{N}$, and via a standard compactness argument,
\[
\bigcap_{N \in \mc{N}}NH = \left(\bigcap_{N \in \mc{N}}N\right)H.
\]
Additionally, since $G$ is compactly generated, $G$ is generated by finitely many cosets of $H$. 

If $G$ is residually discrete, i.e.\ $\bigcap_{N \in \mc{N}}N = \triv$, then $H$ is weakly separable. Being compact and open, $H$ is commensurated, so by the Main Theorem, $H$ contains a normal subgroup $N$ of $G$ that is of finite index in $H$. The closure of $N$ is a compact open normal subgroup of $G$ contained in $H$. Conversely, if $G$ is not residually discrete, then $G$ certainly cannot have a basis of identity neighborhoods consisting of compact open normal subgroups.
\end{proof}

\begin{rmk}
\emph{Unlike the profinite closure, the residual closure is not a closure with respect to a fixed group topology on $G$. However, the residual closure of a \textit{commensurated} subgroup $H$ can be recovered as a closure with respect to some group topology on $G$. Indeed, for $H$ commensurated, there is a canonical locally compact group $\hat{G}_H$ and homomorphism $\alpha: G \rightarrow \hat{G}_H$ such that $\alpha$ has dense image and every finite index subgroup of $H$ is the preimage of a compact open subgroup of $\hat{G}_H$. The group $\hat{G}_H$ is called the Belyaev completion of $(G,H)$; see \cite[\S7]{Belyaev}.  Let $R$ be the intersection of all open normal subgroups of $\hat{G}_H$ and let $\beta: \hat{G}_H \rightarrow \hat{G}_H/R$ be the quotient map.  The subgroup $\wt{H}$ is then the closure of $H$ in the topology induced by $\beta \circ \alpha$ on $G$. In other words, we take the closure in the coarsest group topology on $G$ such that $\beta \circ \alpha$ is continuous.}

\emph{An alternative proof of the Main Theorem can be derived by following this line of reasoning and applying the aforementioned result of Caprace and Monod to the compactly generated totally disconnected group $\hat{G}_H/R$.}
\end{rmk}

\section{Applications to subgroup separability}

We here use the Main Theorem to study separable subgroups. Let us first observe a restatement of the Main Theorem for the profinite closure.

\begin{cor}\label{cor:maintheorem_sep}
Let $G$ be a   group and $H \leq G$ be a commensurated subgroup. Assume that $G$ is generated by finitely many cosets of $H$. Then 
\[
N  := \bigcap_{g \in G} g \overline{H} g^{-1}
\]
is of finite index in the profinite closure $\overline H$ of $H$.  Additionally, $N$ is a normal subgroup of $G$ such that $[H: N \cap H]< \infty$, $N = \overline{N \cap H}$ and $\overline{H} = NH$.
\end{cor}

\begin{proof}
By Lemma~\ref{lem:closure_commensurability}, $L:=\overline{H}$ is a commensurated subgroup. Since finite index subgroups are virtually normal, $\widetilde{L}=L$, so $L$ is weakly separable. Applying the Main Theorem, $N  := \bigcap_{g \in G} g L g^{-1}$ has finite index in $L$. In particular, $[H: N \cap H]< \infty$.

Since $N$ has finite index in $L$ it follows that some \emph{finite} intersection of conjugates of $L$ already coincides with $N$. That is, there are $g_1,\dots,g_n$ such that $N=\bigcap_{i=1}^ng_iLg_i^{-1}$ and therefore that $N=\overline{N}$. In view of Lemma~\ref{lem:N(G,H)}, we may find $M\in \mc{N}(G,G)$ such that $M\cap \overline{H}=N$. Hence, $M\cap H=N\cap H$, and Lemma~\ref{lem:closure_normal} implies that that $\overline{N\cap H}=N$. 

For the final claim, Lemma~\ref{lem:closure_commensurability} ensures that $(\overline{N\cap H})H=\overline H$. By the previous paragraph, we deduce that $NH=\overline H$, completing the proof.
\end{proof}

In view of the Main theorem and Corollary~\ref{cor:maintheorem_sep}, we deduce the following.

\begin{cor}\label{cor:Commens=>Insep}
Let $G$ be a group and $H$ be a  commensurated subgroup such that $G$ is generated by finitely many cosets of $H$.
\begin{enumerate}[label=(\roman{*})]
\item If $H$ is weakly separable, then there is a subgroup $N$ of $H$ that is normal  in $G$ and has finite index in $H$.	
\item If $H$ is separable, then there is a subgroup $N$ that is both normal and separable in $G$ and of finite index in $H$.
\end{enumerate}
\end{cor}

\begin{cor}\label{cor:TrivCore}
Let $G$ be a group and $H$ be a  commensurated subgroup such that $G$ is generated by finitely many cosets of $H$. If $H$ is weakly separable and $\bigcap_{g \in G} g H g^{-1} = \{1\}$, then  $H$ is finite.
\end{cor}

Given $H \leq G$, we denote by $\mathrm{Comm}_G(H)$ the set of those $g \in G$ such that $H$ and $g H g^{-1}$ are commensurate.

\begin{cor}\label{cor:ERF}
Let $G$ be a group. Let  $H ,  J \leq G$ be subgroups such that $J$ is  finitely generated   and that  $J \leq \mathrm{Comm}_G(H)$. If  $H$ is weakly separable in $G$, there exists a subgroup $K \leq H$ that is normal in $\langle J \cup H \rangle$ and of finite index in $H$. 
\end{cor}
\begin{proof}
Since $H$ is weakly separable as a subgroup of $G$, it is also weakly separable as a subgroup of $\langle J \cup H\rangle$. The conclusion follows by applying Corollary~\ref{cor:Commens=>Insep} to the   group $\langle J \cup H\rangle $. 
\end{proof}

\section{Applications to polycyclic groups}

The class of virtually polycyclic groups, often referred to as \emph{polycyclic-by-finite} groups coincides with the class all groups that have a series of finite length with cyclic or finite factor groups. Similarly the class of virtually soluble groups coincides with the class of all groups that have a finite series with abelian or finite factor groups.
The main goal of this section is to prove the following characterizations of virtually polycyclic groups within the class of finitely generated virtually soluble groups. In particular, we will recover the known fact, due to S.~Jeanes and J.~Wilson \cite{JeanesWIlson}, that a finitely generated virtually soluble group in which every subgroup that is subnormal of defect at most $2$ is separable, is virtually polycyclic.  This application grew out of discussions between the second author and B.~Nucinkis.

\begin{theorem}[Jeanes--Wilson, \cite{JeanesWIlson}]\label{cor:polyTFAE}
	The following assertions are equivalent for any  finitely generated virtually soluble group $G$.
	\begin{enumerate}[label=(\roman{*})]
		\item $G$ is virtually polycyclic.
		\item Every subgroup of $G$ is separable.
		\item Every subgroup of $G$ that is subnormal of defect at most $2$ is separable. 
		\item To every $H \leq G$ there is a subgroup $K\le H$ that has finite index in $H$ such that $N_G(K) = \mathrm{Comm}_G(H)$.
		\item For all subgroups $H$ of $G$ that are subnormal of defect at most $2$ and all finitely generated  $J \leq   \mathrm{Comm}_G(H)$, there exists a subgroup $K \leq H$ that has finite index in $H$ and is normal in $\langle J \cup H\rangle$. 
	\end{enumerate} 
\end{theorem}

The fact that (i) implies (iv) is {\cite[(3.1)]{Kropholler}}; however, instead of appealing to \cite{Kropholler}, we remark that this implication can be obtained immediately from Corollary~\ref{cor:Commens=>Insep} and a classical theorem of Mal$'$cev, thus giving a cleaner proof than the original argument of the second author (see the proof of Corollary~\ref{cor:polyTFAE} at the end of this section below).

To explain the proof we first require some background on soluble groups of finite rank. The \textbf{Pr\"ufer rank} of a group is the supremum of the minimum number of generators required for each of its finitely generated subgroups. The \textbf{abelian section rank} is the supremum of the minimum number of generators of the   elementary abelian sections, and it is a theorem of Robinson that finitely generated soluble groups with finite abelian section rank are minimax; see \cite[Theorem 1.1]{Robinson75}. Since the Pr\"ufer rank is bounded below by the abelian section rank, finitely generated virtually soluble groups of finite Pr\"ufer rank are minimax. For these reasons the class of minimax groups inevitably plays a central role in any study of soluble groups and associated finiteness conditions.
Recall that a group $G$ is \textbf{virtually soluble and minimax} provided it has a series
$$\{1\}=G_0\triangleleft G_1\triangleleft\dots\triangleleft G_n=G\eqno\dagger$$
in which the factors are cyclic, quasicyclic, or finite. By a \textbf{quasicyclic} group, we mean a group $C_{p^\infty}$, where $p$ is a prime number, isomorphic to the group of $p$-power roots of unity in the field $\mathbf C$ of complex numbers. For a useful alternative point of view, the exponential map $z\mapsto e^{2\pi iz}$ identifies the additive group $\mathbf Z[\frac1p]/\mathbf Z$ with $C_{p^\infty}$. The terminology \textbf{Pr\"ufer $p$-group} is often used to mean the quasicyclic group $C_{p^\infty}$.

For brevity, we write $\mathfrak M$ for the class of virtually soluble minimax groups.
The following important generalization of Robinson's work on soluble groups of finite rank is crucial to our arguments below.

\begin{theorem}[P.~H.~Kropholler, \cite{Kropholler84}]\label{prop:phk}
Every finitely generated soluble group not belonging to $\mathfrak M$ has a section isomorphic, for some prime $p$, to a lamplighter group $C_p\wr\mathbf Z$.
\end{theorem}

The \textbf{Hirsch length} $h(G)$ of an $\mathfrak M$-group $G$ is defined to be the number of infinite cyclic factors in a cyclic--finite--quasicyclic series witnessing the definition above. The $\mathfrak M$-groups of Hirsch length $0$ satisfy the minimal condition on subgroups as can be seen by a straightforward induction on the length of a quasicyclic--finite series. In fact these Hirsch length $0$ $\mathfrak M$-groups are precisely the \textbf{\v{C}ernikov groups}, each being virtually a direct product of finitely many quasicyclic groups by \v{C}ernikov's theorem \cite[1.4.1]{LennoxRobinson}.

It should be noted that the Hirsch length can be defined for any soluble group $G$ by the formula $h(G):=\sum_{j\ge0}\dim_\Qb G^{(j)}/G^{(j+1)}\otimes\Qb$, and more generally for virtually soluble groups by taking the constant value this formula gives on any subgroup of finite index. For this reason the Hirsch length is sometimes known as the \textbf{torsion-free rank}.

The Fitting subgroup of an $\mathfrak M$-group is always nilpotent, and all $\mathfrak M$-groups are virtually nilpotent-by-abelian. Details of these facts are explained in \cite[\S5.2.2]{LennoxRobinson}.  We refer the reader to \cite[Chapter 5]{LennoxRobinson} for further background information.

By lifting generators of the cyclic sections in a cyclic--finite--quasicyclic series for an $\mathfrak M$-group, we see that every $\mathfrak M$-group contains a finitely generated subgroup with the same Hirsch length. 

\begin{lem}\label{lem:pk1}
Let $G$ be an $\mathfrak M$-group and let $H$ be a subgroup with $h(H)=h(G)$. Then $H$ is separable if and only if $[G:H]<\infty$.
\end{lem}
\begin{proof} The `only if' direction is all that requires proof.

Note first that if $K$ is any normal subgroup of $G$ then
$$h(G)=h(K)+h(G/K)$$
and
$$h(H)=h(H\cap K)+h(H/H\cap K)=h(H\cap K)+h(HK/K).$$
Since $h(H\cap K)\le h(K)$ and $h(HK/K)\le h(G/K)$ we see that equality between $h(H)$ and $h(G)$ also forces both the equalities $h(H\cap K)= h(K)$ and $h(HK/K)= h(G/K)$.

To prove the Lemma we use induction on the length $n$ of a chain $\dagger$ that is witness to $G\in\mathfrak M$. If the length is zero, then $G$ is trivial, and the result is trivially true. Suppose the length $n$ is greater than zero and let $K:=G_{n-1}$ be the penultimate term. The equality $h(H)=h(G)$ forces $h(H\cap K)=h(K)$ and so by induction $H\cap K$ has finite index in $K$. Therefore $H$ has finite index in $HK$, and $HK$ is separable. There are three cases. If $G/K$ is finite, then $H$ has finite index in $G$, and we are done. If $G/K$ is infinite cyclic, then $HK/K$ must also be infinite cyclic, because $H$ has the same Hirsch length as $G$. Therefore $[G:HK]$ is finite, and again we are done. If $G/K$ is quasicyclic, then $HK/K$, being separable in $G/K$, must be equal to $G/K$ (because quasi-cyclic groups do not have any proper finite index subgroups), so $HK=G$.
\end{proof}

\begin{lem}\label{lem:pk2}
Let $G$ be a group in which the finitely generated subgroups are separable. Then every $\mathfrak M$-subgroup of $G$ is virtually polycyclic.
\end{lem}
\begin{proof}
Let $H$ be an $\mathfrak M$-subgroup of $G$ and let $J$ be a subgroup of $H$ that is finitely generated with $h(J)=h(H)$. By Lemma \ref{lem:pk1}, $J$ has finite index in $H$, and hence $H$ is finitely generated. This shows that all $\mathfrak M$-subgroups of $G$ are finitely generated. Since every subgroup of an $\mathfrak M$-group is also a $\mathfrak M$-group it follows that in this situation the $\mathfrak M$-subgroups satisfy the maximal condition on subgroups. Virtually soluble groups with the maximal condition are all virtually polycyclic and the result follows.
\end{proof}

The next lemma follows from \cite[Theorem 3.1]{DFOB} once one interprets the notion of `solvable FAR group' and is there attributed to D. J. S. Robinson. We include a proof for the reader's convenience that is suited to our nomenclature.

\begin{lem}[D. J. S. Robinson]\label{lem:pk3}
Let $G$ be an $\mathfrak M$-group and let $H$ and $K$ be finitely generated subgroups of Hirsch length equal to $h(G)$. Then $H$ and $K$ are commensurate.
\end{lem}
\begin{proof}
The group $\langle H,K\rangle$ has the same Hirsch length as $G$, so replacing $K$ and $G$ by this group, we may assume that $H\subset K=G$ and that $G$ is finitely generated.

Let $h^*$ denote the number of infinite factors in a cyclic--quasicyclic--finite series. We use induction on $h^*(G)$ to prove that $H$ has finite index in $G$. If $h^*(G)=0$, then $G$ is finite, and there is nothing to prove. Let us then assume that $G$ is infinite. Every infinite $\mathfrak M$-group has an infinite abelian normal subgroup. Let $A$ be an infinite abelian normal subgroup of $G$.
Then 
\[
h(HA/A)=h(HA)-h(A)\ge h(H)-h(A)=h(G)-h(A)=h(G/A),
\]
and $h^*(G/A)<h^*(G)$. Therefore $HA$ has finite index in $G$ by induction. We may replace $G$ by $HA$ and so assume that $G=HA$. The intersection $H\cap A$ is normal in $G$, and we have 
\[
h(H/H\cap A)=h(HA/A)=h(G/A)=h(G)-h(A)=h(H)-h(A)
\]
from which it follows that $h(H\cap A)=h(A)$. We deduce that $A/H\cap A$ is torsion. However, $G$ is finitely generated, $A/H\cap A$ is abelian and torsion of finite rank, and $G/H\cap A$ is the semidirect product of $A/H\cap A$ by $G/A$. In a finitely generated semi-direct product, the normal subgroup is always finitely generated as a normal subgroup. Therefore, it  follows that $A/H\cap A$ is finite, so $H$ has finite index in $HA=G$.
\end{proof}

\begin{lem}\label{lem:pk4}
	Let $G$ be a finitely generated virtually soluble group. Assume that for all subnormal subgroups $H$ of $G$ and all finitely generated $J \leq \mathrm{Comm}_G(H)$, there exists a subgroup $K \leq H$ that is normalized by $J$ and has finite index in $H$. Then $G$ is virtually polycyclic. 
\end{lem}
\begin{proof} Suppose first that $G$ is an $\mathfrak M$-group. Let $H$ be a finitely generated subgroup of the Fitting subgroup $N$ of $G$ such that $h(H)=h(N)$. The subgroup $H$ is subnormal, since $N$ is nilpotent, and $\mathrm{Comm}_G(H)=G$, by Lemma~\ref{lem:pk3}. Let $J$ denote the join $H[N,N]$ of $H$ with the commutator subgroup of $N$. Then $J$ also satisfies $\mathrm{Comm}_G(J)=G$ but is now subnormal of defect at most $2$.

By hypothesis, there is a normal subgroup $K$ of $G$ which has finite index in $J$. The quotient $N/K$ is then an $\mathfrak M$-group of Hirsch length $h(N/K)=h(N)-h(K)=0$ and so is a \v{C}ernikov group. Therefore by \v{C}ernikov's theorem \cite[1.4.1]{LennoxRobinson} there is a characteristic abelian subgroup $B/K$ of finite index in $N/K$.  The group $G$ is finitely generated, $B$ is normal in $G$, and $G/B$ is finitely presented. We thus deduce that $B/K$ is finitely generated as a normal subgroup of $G/K$.  It follows that $B/K$ and we deduce that $N/J$ is finite which implies that $N/[N,N]$ is finitely generated.
 Hence $N$ is polycyclic, and $G$ is virtually polycyclic as required.

Let us now suppose toward a contradiction that $G$ is not an $\mathfrak M$-group.
By Proposition~\ref{prop:phk}  (the main theorem of \cite{Kropholler84}), it follows that $G$ has subgroups $K\le J$ such that 
\begin{itemize}
\item $K$ is a normal subgroup of $J$, and
\item  $J/K$ isomorphic to the lamplighter group $C_p\wr \mathbf Z$ for some prime $p$.
\end{itemize} 
The section $J/K$ is a wreath product which can be identified with the matrix group
$$
\left\{
\begin{pmatrix}
t^n&f\\
0&1
\end{pmatrix};\ 
n\in \mathbf Z, f\in\mathbf F_p[t,t^{-1}]
\right\}.$$
Under this identification, the base $B/K$ of the wreath product --- \textbf{the group of lamps} --- corresponds to
$$
\left\{
\begin{pmatrix}
1&f\\
0&1
\end{pmatrix};\ f\in\mathbf F_p[t,t^{-1}]
\right\}.$$
It is clear from the arguments used in \cite{Kropholler84} that $J$ and $K$ can be chosen such that $B/K$ is sandwiched between two terms of the derived series of a soluble subgroup of finite index in $G$. In other words, we may assume that there is a soluble subgroup $G_0$ that is normal and of finite index in $G$ and an $m\ge0$ such that $G_0^{(m+1)}\le K\le B\le G_0^{(m)}$. In particular, both $B$ and $K$ are subnormal in $G$ of defect at most $2$.

Let $H/K$ be the subgroup of the base $B/K$ of consisting of \textbf{half the lamps}, namely the subgroup corresponding to
$$
\left\{
\begin{pmatrix}
1&f\\
0&1
\end{pmatrix};\ f\in\mathbf F_p[t]
\right\}.$$
Note that $H$ is also subnormal of defect at most $2$.
Clearly $\mathrm{Comm}_G(H)$ contains $J$ while the intersection of the conjugates of $H$ is contained in $K$. Let $J_0$ be a finitely generated subgroup of $J$ such that $J_0K=J$. 
Applying the hypothesis to $J_0$ and $H$, we deduce that $J_0$ normalizes some subgroup of finite index in $H_0$, and therefore $J$ normalizes some subgroup of finite index in $H$. This is a contradiction and so excludes the possibility of large wreath product sections in $G$.
\end{proof}

We can now combine the results of this section to prove the main application.

\begin{proof}[Proof of Corollary~\ref{cor:polyTFAE}]
The implication (i) $\Rightarrow$ (ii) is a result of Mal$'$cev~\cite{Malcev}; moreover, every subgroup of a virtually polycyclic group is finitely generated.  Thus we obtain (i) $\Rightarrow$ (iv) as a special case of Corollary~\ref{cor:Commens=>Insep}, by considering $H$ as a subgroup of $\mathrm{Comm}_G(H)$.  Clearly (ii) implies (iii) and (iv) implies (v).  The implication (iii) $\Rightarrow$ (v) is valid in any group by Corollary~\ref{cor:ERF}.  Thus (i) implies all the other assertions, and (v) is implied by each of the other assertions.  Lemma~\ref{lem:pk4} ensures that (v) implies (i), and hence that all five assertions are equivalent.
\end{proof}

\section{Applications to residually finite groups}

Given $H \le G$, we say that $H$ is \textbf{relatively residually finite} in $G$ if the subgroups of $\mc{N}(G,H)$ have trivial intersection. In other words, every non-trivial element of $H$ is separated from the identity by a quotient of $G$ in which $H$ has finite image.  If $G$ itself is residually finite, then every subgroup is relatively residually finite, so the results of this section will apply to commensurated subgroups of residually finite groups.

\begin{lem}\label{lem:FiniteResidual:relative}
Let $G$ be a group and $H \leq G$ be a subgroup. Then 
$$[C_G(H),  \widetilde{H}] \subseteq \bigcap_{N \in \mc{N}(G,H)} N 
\hspace{1cm} \text{and} \hspace{1cm}
[C_G(H),  \overline{H}] \subseteq \bigcap_{N \in \mc{N}(G,G)} N.$$ 
In particular, if $H$ is  a relatively residually finite subgroup, then $C_G(H) = C_G(\widetilde{H})$, and if  $G$ is residually finite, then   $C_G(H) = C_G(\overline{H})$.
\end{lem}

\begin{proof}
Let $x \in C_G(H)$ and $y \in \widetilde{H}$.  Then for all $N \in \mc{N}(G,H)$ we have $y \in HN$, and hence $[x,y] \in N$.  The required conclusion follows, observing in addition that the inclusion $C_G(H) \geq  C_G(\widetilde{H})$ is obvious.  

The proof of the corresponding fact about the profinite closure is the same.
\end{proof}

The \textbf{FC-centralizer} of a subgroup $H$ of a group $G$, denoted by $\FC_G(H)$, is the collection of those elements $g \in G$ which centralize a finite index subgroup of $H$. The FC-centralizer $\FC_G(H)$ is a normal subgroup of the commensurator $\mathrm{Comm}_G(H)$. 
Notice moreover that the FC-centralizer $\FC_H(H)$ coincides with the \textbf{FC-center} of $H$, i.e. the set of elements of $H$ whose $H$-conjugacy class is finite. A group $G$ is called an $\textbf{FC-group}$ if $G = \FC_G(G)$ or, equivalently, if all elements of $G$ have a finite conjugacy class. We underline the difference between an \textbf{FC-subgroup} of $G$, which  is a subgroup $H$ such that $\FC_H(H) = H$, and an \textbf{FC-central subgroup} of $G$, which is a subgroup of $\FC_G(G)$. 

\begin{cor}\label{cor:NbC:tech}
Let $G$ be a group, let $N$ be a normal subgroup and $H$ be a commensurated  subgroup of $G$ such that $N \cap H = \{1\}$. Assume that every normal FC-subgroup of $G$ is finite. If $G$ is generated by finitely many cosets of $H$ and if $H$ is relatively residually finite in $G$, then $H$ has a finite index subgroup that commutes with $N$. 
\end{cor}

\begin{proof}
For $x \in  N$, there is a finite index subgroup $H_1 \leq H$ such that $x H_1 x^{-1} \leq H$. The commutator $[x, H_1]$ is contained in the intersection $H \cap  N  =\{1\}$, so $x \in C_G(H_1)$. We deduce from Lemma~\ref{lem:FiniteResidual:relative} that $x \in C_G(\widetilde{H_1})$ since $H$ is relatively residually finite. Moreover, by Lemma~\ref{lem:closure_commensurability}, the index of $\widetilde{H_1}$ in $\widetilde H$ is finite. This  shows that $N \leq \FC_G(\widetilde H)$. Let $M$ be the   normal subgroup of $G$ obtained by applying the Main Theorem to $H$, so that $M$ is a finite index subgroup of $\widetilde H$. Thus  we have $N \leq \FC_G(M)$. In particular $N \cap M$ is contained in $\FC_M(M)$, which is a normal FC-subgroup of $G$. By hypothesis, it is finite, so that $N \cap M$ is finite. By Lemma~\ref{lem:closure_normal}, we have $\mathcal N(G,H) = \mathcal N(G,\widetilde H)$. Since $H$ is relatively residually finite, it follows that $\widetilde H$, and hence also $M$, are relatively residually finite. Therefore, since $N \cap M$ is finite,  there exists $Q \in \mathcal N(G,M)$ such that $N \cap Q = \{1\}$. Since $N$ and $Q$ are both normal, they commute. Thus $H \cap Q$ has finite index in $H$ and commutes with $N$.  
\end{proof}

\begin{cor}\label{cor:NbC}
Let $G$ be a finitely generated residually finite group all of whose  amenable normal subgroups are finite. For any normal subgroup $N$ and any commensurated subgroup $H$, if $N \cap H = \{1\}$, then some subgroup of finite index in $H$ commutes with $N$. 
\end{cor}
\begin{proof}
Every FC-group is \{locally finite\}-by-abelian, see \cite[Theorem~5.1 and Corollary~5.13]{Neumann}. In particular FC-groups are amenable. Thus the required conclusion  follows directly from Corollary~\ref{cor:NbC:tech}. 
\end{proof}

In \cite{BFS}, U. Bader, A. Furman and R. Sauer have undertaken a systematic study of \textbf{lattice envelopes} of an abstract group $\Gamma$, that is, a description of the structure of all locally compact groups $G$ that   contain an isomorphic copy of 
$\Gamma$ as a lattice. Their theory requires the abstract group $\Gamma$ to satisfy three conditions. One of those conditions is that for any normal subgroup N and any commensurated subgroup H in $\Gamma$, if $N\cap H = \{1\}$, then H has a finite index subgroup that commutes with $N$. Thus Corollary~\ref{cor:NbC} shows that this one of the Bader--Furman--Sauer conditions is automatically satisfied by every finitely generated residually finite group whose amenable radical is finite.  

\begin{cor}\label{cor:RF:relative}
Let $G$ be a finitely generated group in which every infinite normal subgroup has trivial centralizer. Then every infinite commensurated relatively residually finite subgroup has trivial FC-centralizer. 
\end{cor}
\begin{proof}
Let $H \leq G$ be an infinite commensurated subgroup and let $x \in \FC_G(H)$. There is a finite index subgroup $H_0 \leq H$ such that $x \in C_G(H_0)$. Let then $N$ be the normal subgroup of $G$ obtained by applying the Main Theorem to $H_0$ and note that $N$ is relatively residually finite as a consequence of  Lemma~\ref{lem:closure_normal}.  Since $N = \widetilde{N \cap H_0}$, we deduce from Lemma~\ref{lem:FiniteResidual:relative} that $C_G(N) = C_G(N \cap H_0)$. The group $N$ is infinite, so $C_G(N)$ is trivial by hypothesis. Hence, $x \in C_G(H_0)\leq C_G(N \cap H_0)$ is trivial.  
\end{proof}

The hypothesis that $H$ be relatively residually finite cannot be removed in Corollary~\ref{cor:RF:relative}. As we shall see in the next section, this is illustrated by the Baumslag--Solitar groups. 

\begin{lem}\label{lem:finite_normal}
Let $G$ be an infinite group in which every infinite normal subgroup has trivial centralizer.  Then every non-trivial normal subgroup is infinite.
\end{lem}

\begin{proof}
Let $N$ be a finite normal subgroup of $G$.  Its centralizer $C_G(N)$ is a normal subgroup of finite index, hence it is infinite. On the other hand, $C_G(C_G(N))$, which contains $N$, is trivial by hypothesis.
\end{proof}

\begin{cor}\label{cor:Intersections:relative}
Let $G$ be a finitely generated group in which every infinite normal subgroup has trivial centralizer.  Then every infinite commensurated relatively residually finite subgroup has an infinite intersection with every non-trivial normal subgroup.
\end{cor}
\begin{proof}
We assume that $G$ is infinite.  Let $H \leq G$ be an infinite commensurated subgroup and $N \leq G$ be a non-trivial normal subgroup. By Lemma~\ref{lem:finite_normal}, $N$ is infinite.
	
Assume toward a contradiction that  $H \cap N$ is finite. Since $H$ is infinite and relatively residually finite, there is an infinite $M \in \mc{N}(G,H)$ such that $H \cap N \cap M = \{1\}$. For $x \in  N \cap M$, there is a finite index subgroup $H_1 \leq H$ such that $x H_1 x^{-1} \leq H$. The commutator $[x, H_1]$ is contained in the intersection $H \cap  N \cap M =\{1\}$, so $x \in C_G(H_1)$. We conclude that $ N \cap M  \leq \FC_G(H)$, so $ N \cap M  = \{1\}$ since $\FC_G(H) = \{1\}$ by Corollary~\ref{cor:RF:relative}. In particular, $N \leq C_G(M)=\{1\}$, which is absurd.
\end{proof}

The conclusion of Corollary~\ref{cor:Intersections:relative} cannot be extended to a conclusion that any two infinite commensurated relatively residually finite subgroups have infinite intersection.  For instance, let $\Gamma < G_1 \times G_2$ be an irreducible residually finite lattice in a product of two totally disconnected locally compact groups and choose $U_1,U_2$ to be compact open subgroups of $G_1,G_2$ respectively such that $\Gamma \cap U_1 \times U_2 = \triv$.  We then obtain two infinite commensurated subgroups of $\Gamma$ with trivial intersection, namely $W_1 := \Gamma \cap (U_1 \times G_2)$ and $W_2 := \Gamma \cap (G_1 \times U_2)$.

To conclude this section, we note a property of residually finite dense subgroups of totally disconnected locally compact groups.

\begin{cor}\label{cor:no_discrete_normal}
Let $G$ be a non-discrete totally disconnected locally compact group such that every infinite normal subgroup of $G$ has trivial centralizer.  If $\Gamma$ is a dense subgroup of $G$ that is finitely generated and residually finite, then the only discrete normal subgroup of $G$ contained in $\Gamma$ is the trivial subgroup.
\end{cor}

\begin{proof}
Given an infinite normal subgroup $M$ of $\Gamma$, we see that the closure of $M$ is an infinite normal subgroup of $G$ and thus has trivial centralizer.  Since the centralizer is unaffected by taking the closure, it follows that $C_G(M) = \triv$.  W conclude that $\Gamma$ satisfies the hypotheses of Corollary~\ref{cor:Intersections:relative}.

Let $N$ be a non-trivial normal subgroup of $\Gamma$, let $U$ be a compact open subgroup of $G$ and set $H = \Gamma \cap U$.  The subgroup $H$ is commensurated in $\Gamma$, so by Corollary~\ref{cor:Intersections:relative}, the intersection $N \cap H$ is infinite. It now follows that $N$ is not discrete.
\end{proof}

\section{Applications to generalized Baumslag--Solitar groups}

In the setting of Hausdorff topological groups, the collection of elements which satisfy a fixed law is often closed. Alternatively, if a set satisfies a law, then so does its closure. The simplest example of this phenomenon is that centralizers are always closed in a Hausdorff topological group. 

While the residual closure does not necessarily come from a group topology, it does appear to behave well with respect to laws; cf. Corollary~\ref{lem:FiniteResidual:relative}. We here explore the extent to which laws pass to the the residual closure of a subgroup.

\begin{lem}\label{lem:pkvar}
Let $\mathcal V$ be a variety of groups. If $G$ is a group and $H$ is a relatively residually finite $\mathcal V$-subgroup of $G$, then the residual closure $\wt{H}$ of $H$ also belongs to $\mathcal V$.  If $G$ is residually finite, then $\overline{H} \in \mathcal{V}$. 
\end{lem}

\begin{proof}
We see from the hypotheses that there is an injective map from $\wt{H}$ to a profinite group $K$, where $K$ is the inverse limit of the finite groups $HN/N$ for $N \in \mc{N}(G,H)$, such that the image of $H$ in $K$ is dense.  Since $H$ is a $\mathcal{V}$-group, it follows that $K$ is a $\mathcal{V}$-group and hence $\wt{H}$ is a $\mathcal{V}$-group.  If $G$ is residually finite, the argument that $\overline{H} \in \mathcal{V}$ is similar.
\end{proof}

\begin{cor}\label{cor:pk4.1}
Let $\mathcal V$ be a variety of groups and let $G$ be a finitely generated group. Suppose that $H$ is a commensurated relatively residually finite $\mathcal V$-subgroup of $G$.
Then there is a normal $\mathcal V$-subgroup $K$ of $G$ such that $[H:H\cap K]<\infty$. 
\end{cor}

We next consider the \textbf{Baumslag--Solitar groups} $\mathrm{BS}(m,n)$. The group $\mathrm{BS}(m, n)$ is the one-relator group given by the presentation
\[
 \langle a, t \mid t a^m t^{-1} = a^n\rangle,
 \]
where  $m$ and $n$ are integers with $mn\ne0$. A Baumslag--Solitar group is an HNN-extension of $\mathbb{Z}$ and so acts on the associated Bass--Serre tree.

We can now provide a swift strategy for recovering the known result on residual finiteness of Baumslag--Solitar groups, predicted in the original work of Baumslag and Solitar and subsequently established by Meskin.

\begin{theorem}[Meskin,\ \cite{Meskin}]\label{cor:pk4.2}
The Baumslag--Solitar group $\mathrm{BS}(m,n)$ is residually finite if and only if the set $\{1,|m|,|n|\}$ has at most $2$ elements.
\end{theorem}
\begin{proof}
We focus on one direction; namely, if $\mathrm{BS}(m,n)$ is residually finite, then 
$$| \{1,|m|,|n|\}|\le2.$$ This is the implication where our Main Theorem provides a significant insight. The converse is rather more routine and we are not here adding anything to the argument that can be found in \cite{Meskin}.

We prove the contrapositive. Suppose that $|\{1,|m|,|n|\}|>2$ and set $G:=\mathrm{BS}(m,n)$. The cyclic subgroup $\langle a\rangle$ fixes a vertex of the Bass--Serre tree $T$ and the action is vertex transitive. Since $|m|\ne|n|$, the subgroup $\bigcap_{g\in G}g\langle a\rangle g^{-1}$ is trivial, and thereby, the representation $G\to\mathrm{Aut}(T)$ is faithful. The vertex stabilizers of $G$ also do not fix any of the incident edges, so in particular, the action of $G$ on $T$ does not fix an end. 

In any group acting minimally without a fixed end on a tree with more than two ends, every normal subgroup either acts trivially or acts minimally without a fixed end, by \cite[Lemme 4.4]{Tits2}. Since a subgroup of $\Aut(T)$ acting minimally has trivial centralizer in $\Aut(T)$ (because the displacement function of any element in the centralizer is constant), it follows that  every non-trivial normal subgroup of $G$ has trivial centralizer. On the other hand, $\langle a\rangle$ is an infinite commensurated abelian subgroup of $G$, which is thus contained in its own FC-centralizer.  

We conclude from Corollary~\ref{cor:RF:relative} that $G$ is not residually finite in this case, and indeed that $\langle a\rangle$ is not even relatively residually finite in $G=\mathrm{BS}(m,n)$.
\end{proof}

As a further illustration of these ideas we offer an application to certain fundamental groups of graphs of virtually soluble groups that generalizes some of the aspects of Meskin's result (Corollary \ref{cor:pk4.2}).  In particular, the class we consider includes all \textbf{generalized Baumslag--Solitar groups}, that is, fundamental groups of graphs of cyclic groups.  We shall need two lemmas in preparation.

\begin{lem}\label{lem:pk24.1}
Let $B$ be a group, let $\phi \in \mathrm{End}(B)$ and let $B_{\infty} = \lim_\to(B\xrightarrow{\phi} B \xrightarrow{\phi} \dots)$.
\begin{enumerate}
\item If $B$ is virtually soluble then $B_{\infty}$ is virtually soluble.
\item The Pr\"ufer rank of $B_{\infty}$ is at most the Pr\"ufer rank of $B$.
\item If $B$ belongs to $\mathfrak M$, then $B_{\infty}$ belongs to $\mathfrak M$, with $h(B_{\infty}) \leq h(B)$.
\end{enumerate}
\end{lem}
\begin{proof}
We have a sequence $(\tau_i)_{i \ge 0}$ of homomorphisms $\tau_i: B \rightarrow B_{\infty}$ arising from the direct system, such that $\tau_i = \tau_{i+1} \circ \phi$ and $B_{\infty} = \bigcup_{i \ge 0}\tau_i(B)$.  More generally, given $C \subset B$ such that $\phi(C) \subset C$, write $C_{\infty}$ for the ascending union $\bigcup_{i \ge 0}\tau_i(C)$.

\begin{enumerate}
\item
Let $C$ be a soluble normal subgroup of finite index in $B$ and let $k$ be the maximal derived length of a soluble subgroup of $B$.  We now show by induction that for each $m\ge0$, 
$$D_m:=C\phi(C)\phi^2(C)\dots \phi^m(C)$$ is a soluble subgroup of $B$ of derived length at most $k$.
This is true when $m=0$ by the choice of $C$.
Suppose now that $m>0$ and, inductively, that $D_{m-1}$ is a soluble subgroup of $B$. Then $\phi(D_{m-1})\le \phi(B)\le B$, so $D_m=C\phi(D_{m-1}) \le B$. Moreover, $C$ is normal and soluble, $D_{m-1}$ is soluble by induction, and so $D_m$ is a soluble subgroup of length at most $k$.  We thus have an ascending chain of soluble subgroups
$C=D_0\le D_1\le D_2\le\dots$ each of length at most $k$ and having finite index in $B$. Let $D$ denote the union $\bigcup_jD_j$; of course, $D=D_m$ for all sufficiently large $m$.

By construction, $\phi(D) \le D$, so we may form the ascending union $D_{\infty}$. The group $D_{\infty}$ belongs to the variety of $D$ and hence is soluble of length at most $k$. Let $n$ denote the (finite) index of $D$ in $B$. For each $j$, we see that 
$$\qquad\ \ [\tau_j(B)D_\infty:D_\infty]=[\tau_j(B):\tau_j(B)\cap D_\infty]\le[\tau_j(B):\tau_j(D)]\le[B:D]=n,$$ and so 
we deduce that $[B_\infty:D_\infty]\le n$.  Thus, $B_{\infty}$ is virtually soluble.

\item It is clear that $\tau_i(B)$ has at most the Pr\"ufer rank of $B$ for each $i$, and the Pr\"ufer rank of $B_{\infty}$ is the supremum of the Pr\"ufer ranks of $\tau_i(B)$.

\item
Consider first the case when $B$ is torsion-free abelian. In this case, $\phi$ induces a $\Qb$-linear map $\theta: B\otimes\Qb\to B\otimes\Qb$.  Since $B$ has finite Hirsch length, $B \otimes \Qb$ is finite-dimensional over $\Qb$.  By replacing $B$ with $\phi^i(B)$ for sufficiently large $i$, we may assume $h(\phi(B)) = h(B)$ and hence ensure that $\theta$ has full rank, in other words $\theta$ is an automorphism.  Upon choosing a basis $v_1,\dots,v_d$ of $B\otimes\Qb$, we obtain a matrix $\Theta$ corresponding to $\theta$ whose entries, being finite in number, belong to a subring $\Zb[1/n]$ for a choice of common denominator $n$. Let $m$ be the product of the finitely many primes $q$ for which $B$ has Pr\"ufer $q$-group $C_{q^\infty}$ as a section.
It is now clear that $\theta$ induces an automorphism of $B\otimes\Zb[1/nm]$ 
and that $B\otimes\Zb[1/nm]\cong B_\infty\otimes\Zb[1/nm]$ is a free $\Zb[1/nm]$-module of rank $d$ and belongs to $\mathfrak M$. This shows that $B_\infty \in \mathfrak M$ and also that $h(B_{\infty}) \le h(B)$.

Suppose that $B$ is an abelian $\mathfrak M$-group with torsion subgroup $T$. Then $\phi(T)$ is torsion, so $\phi(T) \le T$.  Torsion $\mathfrak M$-groups satisfy the minimal condition on subgroups, so $\phi$ restricts to a surjective map on $\phi^j(T)$ for some $j$.  Then 
$$\forall k \ge 0: \; \tau_0(T) = \tau_{j+k}(\phi^{j+k}(T)) = \tau_{j+k}(\phi^j(T)) = \tau_k(T),$$
so we see that $B_\infty$ is $\tau_0(T)$ by the direct limit of iterating an endomorphism of $B/T$.  From the torsion-free case, we conclude that $B_{\infty} \in \mathfrak M$.

For the case when $B$ is a soluble $\mathfrak M$-group, we proceed by induction on the derived length.  Since the commutator subgroup is verbal, we see that $[B_{\infty},B_{\infty}] = [B,B]_{\infty}$. Our inductive hypothesis implies that $[B_{\infty},B_{\infty}]$ is an $\mathfrak M$-group, and on the other hand, $B_{\infty}/[B_{\infty},B_{\infty}]$ is the direct limit of iterating the endomorphism of $B/[B,B]$ induced by $\phi$. We deduce that $B_{\infty}/[B_{\infty},B_{\infty}]$ is an $\mathfrak M$-group by the abelian case, and hence $B_{\infty} \in \mathfrak M$.  A similar induction argument on the derived length shows that $h(B_{\infty}) \le h(B)$.  

Finally, consider the general virtually soluble case. The argument used to prove (1) shows that there is a soluble subgroup $D$ of finite index in $B$ such that $\phi(D) \le D$. We know that $D_\infty \in \mathfrak M$ and $h(D_{\infty}) \le h(D)$ by the soluble case, and the result follows since $D_\infty$ has finite index in $B_\infty$.
\end{enumerate}
\end{proof}

\begin{lem}\label{lem:pk24.2}
Let $G$ be a finitely generated group acting on a tree $T$ in such a way that there is no global fixed point and there is a unique fixed end. Then there is a vertex $u$ and a hyperbolic element $t\in G$ such that $G_u\subset tG_ut^{-1}$ and $G$ is the ascending HNN-extension $G_u{*}_{G_u,t}$.
\end{lem}
\begin{proof}
We can define a partial ordering $\to$ on the set of vertices of $T$ by declaring that $v\to w$ when the geodesic ray starting from $v$ and traveling towards the fixed end passes through $w$. This makes the vertex set into a directed set because the rays from any two vertices that head towards the fixed end eventually coalesce and so reach a vertex to which they both point. Note also that if $v\to w$ then $G_v\subset G_w$. The set $H$ of elliptic elements of $G$ is thus the directed union of the vertex stabilizers. If $H=G$, then $H$ is finitely generated and so is contained in a vertex stabilizer, and this contradicts the assumption that $G$ has no global fixed point. Therefore $H\ne G$ and $G$ contains a hyperbolic element $t$. Let $L$ be the axis of $t$ and fix a vertex $u$ on $L$. Replacing $t$ by $t^{-1}$ if necessary we may assume that $u\to tu$. If $h$ is any element of $H$, then $H$ fixes a vertex $w$, and there is a $j\ge0$ such that $w\to t^j u$. This shows that $H=\bigcup_jt^jG_ut^{-j}$, and the result is clear.
\end{proof}

\begin{prop} Let $G$ be a finitely generated residually finite group that is the fundamental group of a graph of virtually soluble groups of finite Pr\"ufer rank and of Hirsch length $n$. Then $G$ has a soluble normal subgroup $N$ such that one of the following holds.
\begin{enumerate}
\item $N$ fixes points on the corresponding Bass--Serre tree of $G$ and has Hirsch length $n$, and $G/N$ is the fundamental group of a graph of locally finite groups. If in addition all the vertex and edge groups of the graph of groups belong to $\mathfrak M$, then $G/N$ is virtually free.
\item $N$ has no fixed points on the corresponding Bass--Serre tree of $G$, $G$ fixes a unique end and is a virtually soluble $\mathfrak M$-group of Hirsch length $n+1$, and $G/N$ is infinite and virtually cyclic. Note that in this case all vertex and edge stabilizers automatically belong to $\mathfrak M$ and that $G/N$ is virtually free of rank $1$.
\end{enumerate}
\end{prop}

\begin{proof}
Let $T$ denote the Bass--Serre tree $T$ afforded by the hypothesized graph of groups. For each vertex or edge $y\in VT\sqcup ET$, let $H_y$ be a finitely generated subgroup of the stabilizer $G_y$ of Hirsch length $n$. Note first that if $e$ and $f$ are edges that are incident with the same vertex $v$, then $\langle H_e\cup H_f\rangle$ is a finitely generated subgroup of $G_v$, and therefore by Lemma \ref{lem:pk3}, $\langle H_e\cup H_f\rangle$ is commensurate with $H_v$. Lemma \ref{lem:pk3} also shows that $H_e$ and $H_f$ have finite index in
$\langle H_e\cup H_f\rangle$ and so we deduce that $H_e$, $H_v$ and $H_f$ are all commensurate. Suppose $v$ and $w$ are any two vertices in $T$. By considering the edges and vertices on the geodesic from $v$ to $w$, we deduce that $H_v$ and $H_w$ are commensurate. 

Fix any vertex $v$ and let $H$ be a soluble subgroup of finite index in $H_v$.
The profinite closure $\overline H$ of $H$ in $G$ is  soluble, with the same derived length as $H$, by Lemma~\ref{lem:pkvar}. Additionally, $\overline H$ has a finite index subgroup $N$ which is a separable normal subgroup of  $G$ by Corollary~\ref{cor:Commens=>Insep}. The group $N$ is a soluble normal subgroup of $G$, and for every vertex or edge $y$ of the tree, the intersection $N \cap G_y$ has Hirsch length  $n = h(G_y)$ and the quotient $G_y/N\cap G_y$ is locally finite, residually finite, and of finite Pr\"ufer rank. 

A soluble group acting without inversion on a tree fixes a point or a unique end or stabilizes a unique pair of ends; see for example \cite[Corollary~2]{Tits_LieKolchin}. Suppose toward a contradiction that there is a unique pair of ends fixed by $N$ but that $N$ does not fix any vertex or edge. In this case, $N$ stabilizes the line joining the two ends and so does $G$. There is then a homomorphism $\xi$ from $G$ to an infinite cyclic or dihedral group induced by the action of $G$ on the line.  On the other hand, $H$ acts on the line as a group of order at most $2$, since $H$ fixes a vertex.  By replacing $H$ with a subgroup of finite index as necessary, we can ensure that $H\leq \ker\xi$. The quotient $G/\ker\xi$ is residually finite, so the profinite closure of $H$  and hence also $N$ act trivially on the line. This contradicts the assumption that $N$ has no fixed points.

Suppose that $N$ has a fixed point in $T$. In this case $G/N$ acts on the subtree $T^N$ of $N$-fixed points and is the fundamental group of a graph of locally finite groups as claimed. For any vertex $u$ of $T^N$, we additionally have $N\subset G_u$. Thus, $N\cap G_u=N$, so $N$ has Hirsch length $n$.  Moreover, the quotient $G_u/N$ is a residually finite and locally finite $\mathfrak{M}$-group, so  \v{C}ernikov's theorem \cite[1.4.1]{LennoxRobinson} ensures that $G_u/N$ is finite. The quotient $G/N$ is then a graph of finite groups and thus is a virtually free group since $G$ is finitely generated.  We thus obtain case (1).

Suppose that $N$ fixes a unique end of $T$ but does not fix any vertices. In this case $G$ has the same property, and Lemma \ref{lem:pk24.2} shows that $G$ is an ascending HNN-extension over one of its vertex stabilizers $G_y$. The group $G$ is thus of the form $B_{\infty} \rtimes \Zb$ where $B_{\infty}$ is a virtually soluble group of finite Pr\"ufer rank and Hirsch length $n$ by Lemma~\ref{lem:pk24.1}. Thus, $G$ has Hirsch length $n+1$ and is virtually soluble. Since $G$ is also assumed to be finitely generated, it belongs to $\mathfrak M$, by Proposition~\ref{prop:phk}. Finally, the quotient $G_y/N\cap G_y$ is an $\mf{M}$-group as well as residually finite and locally finite. Any residually finite and torsion $\mathfrak M$-group  is finite by \v{C}ernikov's theorem \cite[1.4.1]{LennoxRobinson}, so $G_y/N_y$ is finite. Since $N$ is normal, $\bigcup_{t\geq 1} t^nN_yt^{-n}$ is a subgroup of $N$ where $t$ is the generator of $\Zb$ that translates toward the fixed end. Just as in the proof of Lemma~\ref{lem:pk24.1}, we deduce that $B_{\infty}/\bigcup_{n\geq 1}t^nN_yt^{-n}$ is finite, and thus, $G/N$ is finite-by-cyclic. It follows that $G/N$ is virtually cyclic, so we obtain case (2).
\end{proof}

\section{Applications to groups acting on graphs and metric spaces}
Commensurated subgroups arise naturally from actions on locally finite graphs. We thus obtain several consequences for such group actions. 

\begin{cor}\label{cor:Graphs}
	Let $X$ be a connected locally finite graph (without multiple edges) and $G \leq \Aut(X)$ be a finitely generated vertex-transitive group.  If the stabilizer $G_v$ of a vertex $v \in VX$ is weakly separable, then it is finite.
\end{cor}
\begin{proof}
	Since $G$ acts vertex-transitively, we have $\{1\} =\bigcap_{w \in VX} G_w = \bigcap_{g \in G} g G_v g^{-1}$. On the other hand, the stabilizer $G_v$ is a commensurated subgroup  of $G$  since $X$ is connected and locally finite. The conclusion now follows from Corollary~\ref{cor:TrivCore}.
\end{proof}

\begin{cor}\label{cor:Trees}
	Let $X$ be a proper uniquely geodesic metric space (e.g. a locally finite tree) and let $G \leq \mathrm{Isom}(X)$ be a finitely generated group. Suppose that there is no proper $G$-invariant convex subspace. If the orbit of a point $v \in X$ is discrete and its stabilizer $G_v$ is weakly separable, then $G_v$ is finite.
\end{cor}

\begin{proof}
The hypothesis of absence of  proper $G$-invariant convex subspace implies that for any vertex $v \in X$, the intersection $\bigcap_{g \in G} g G_v g^{-1}$ is trivial, because it fixes pointwise a $G$-invariant subspace, namely the convex hull of the $G$-orbit of $v$.  The hypothesis that $X$ is proper and that $Gv$ is discrete ensures that $G_v$ has finite orbits on $Gv$, so $G_v$ is commensurated.  The conclusion now follows from Corollary~\ref{cor:TrivCore}.
\end{proof}


The terminology in the following application is borrowed from \cite{BuMo}.

\begin{cor}\label{cor:LocallyQuasiPrimitive}
	Let $T$ be a locally finite tree all of whose vertices have degree~$\geq 3$ and let $G \leq \Aut(T)$ be a non-discrete finitely generated subgroup whose action on $T$ is locally quasi-primitive. For $v \in VT$ a vertex, the residual closure of the vertex stabilizer $G_v$ in $G$ is of finite index in $G$.
\end{cor}
\begin{proof}
	Since $T$ is locally finite, the vertex stabilizer $H := G_v$ is commensurated. The group $H$ is also infinite since $G$ is non-discrete. Let $N := \bigcap_{g \in G} g \widetilde H g^{-1} \normal G$  be the normal subgroup afforded by the Main Theorem.  
	Since $N$ contains a finite index subgroup of $H$, its action on $VT$ is not free. Therefore, \cite[Lemma~1.4.2]{BuMo} implies that the  $N$-action on $T$ has finitely many orbits of vertices.   In particular, $N H = \widetilde{H}$ is of finite index in $G$. 
\end{proof}	

\section{Applications to lattices in products of groups}
Lattices in products of totally disconnected locally compact groups often have interesting commensurated subgroups. We here apply our work to shed light on these subgroups.

The equivalence between (i) and (iv) in the following result is due to M.~Burger and S.~Mozes  \cite[Proposition~1.2]{BuMo2}. Our inspiration for the equivalence between (ii) and (iv) came  from contemplating D. Wise's iconic example constructed in \cite[Example~4.1]{Wise}; see also \cite{Wise_PhD}. The equivalence between  (iii) and (iv) is closely related to, but not a  formal consequence of, another result of Wise \cite[Lemmas~5.7 and~16.2]{Wise_Figure8}. 

\begin{cor}\label{cor:ProductOIfTrees}
	Let $T_1 , T_2$ be leafless trees and let $\Gamma \leq \Aut(T_1) \times \Aut(T_2)$ be a discrete subgroup acting cocompactly on $T_1 \times T_2$. Then the following assertions are equivalent. 
	\begin{enumerate}[label=(\roman*)]
		\item There exists  $i \in \{1, 2\}$ such that the projection $\mathrm{pr}_i(\Gamma) \leq \Aut(T_i)$ is discrete. 
		
		\item There exists $i \in \{1, 2\}$ and a vertex $v \in VT_i$ such that the stabilizer $\Gamma_{v}$ is a weakly separable subgroup of $\Gamma$.
		
		\item For all $i \in \{1, 2\}$ and all $v \in VT_i$, the stabilizer $\Gamma_{v}$ is a separable subgroup of $\Gamma$.
		
		\item The groups $\Gamma_1 = \{g \in \Aut(T_1) \mid (g,1) \in \Gamma\}$ and $\Gamma_2 = \{g \in \Aut(T_2) \mid (1, g) \in \Gamma\}$ act cocompactly on $T_1$ and $T_2$ respectively, and the product $\Gamma_1 \times \Gamma_2$ is of finite index in $\Gamma$.
	\end{enumerate}
\end{cor}

\begin{proof}
For the equivalence of (i) and (iv), see \cite[Proposition~1.2]{BuMo2}.  If (iv) holds, then $\Gamma$ is virtually the product of two groups $\Gamma_1$ and $\Gamma_2$ acting properly and cocompactly on $T_1$ and $T_2$ respectively. Such groups are virtually free, hence residually finite, so $\Gamma$, $\Gamma/\Gamma_1$ and $\Gamma/\Gamma_2$ are all residually finite. In particular, $\Gamma_1$ and $\Gamma_2$ are separable in $\Gamma$. For every $v \in VT_i$, the group $\Gamma_v$ is separable, since $\Gamma_v/\Gamma_{3-i}$ is finite. Assertion (iv) thus implies (iii).

That (iii) implies (ii) is clear. That  (ii) implies (i) follows from Corollary~\ref{cor:Trees}, using the fact that a cocompact action on a leafless tree does not preserve any proper subtree.
\end{proof}

We obtain the following abstract generalization of a statement originally proved by Burger and Mozes for lattices in products of trees with locally quasi-primitive actions (see  \cite[Proposition~2.1]{BuMo2}), and later extended to lattices in products of CAT($0$) spaces (see \cite[Proposition~2.4]{CaMoKM}).

\begin{cor}\label{cor:FaithfulProj}
Let $\Gamma$ be a lattice in the product $G_1 \times G_2$ of two locally compact groups. Assume that $G_1$ is totally disconnected and non-discrete and that every infinite closed normal subgroup of $G_1$ has trivial centralizer in $G_1$.  Suppose further that $\Gamma$ is finitely generated and that the canonical projection $\Gamma \to G_1$ has a dense image. If $\Gamma$ is residually finite, then the projection $\Gamma \to G_2$ is injective. 
\end{cor}

\begin{proof}
Let $\proj_1\colon \Gamma \to G_1$ and $\proj_2\colon\Gamma \to G_2$ be the projection maps, let $N_1 := \ker\proj_2 \le G_1$ and let $N_2 := \ker\proj_1 \le G_2$.  We must show that $N_1$ is trivial.  We shall proceed by contradiction and assume that $N_1$ is non-trivial.  The group $N_1$ is a non-trivial discrete subgroup of $G_1$ and is normalized by the dense subgroup $\Gamma_1:= \proj_1(\Gamma)$ of $G_1$.  As $N_1$ is also closed, it is normal in $G_1$, and hence $N_1$ is infinite by Lemma~\ref{lem:finite_normal}.

The groups $N_1$ and $N_2$ are two normal subgroups of $\Gamma$ with trivial intersection, hence they commute. Since $\Gamma$ is residually finite, the profinite closures $\overline{(N_1)}_{\Gamma}$ and $\overline{(N_2)}_{\Gamma}$ also commute, in light of Corollary~\ref{lem:FiniteResidual:relative}. We have in particular that 
$$\proj_1(\overline{(N_2)}_{\Gamma}) \leq C_{G_1}(\proj_1(N_1))=\{1\},$$ 
so $\overline{(N_2)}_{\Gamma} \le \ker\proj_1 = N_2$.  Thus, $N_2$ is profinitely closed in $\Gamma$. Hence, $\Gamma_1 \cong \Gamma/N_2$ is residually finite. We thus contradict Corollary~\ref{cor:no_discrete_normal}, since $N_1$ is an infinite discrete normal subgroup of $\Gamma_1$.
\end{proof}

\section{Applications to just-infinite groups}
We finally consider just-infinite groups. An infinite group is \textbf{just-infinite} if every proper quotient is finite. These groups have restricted normal subgroups. We here explore restrictions on their commensurated subgroups.

\begin{cor}\label{cor:ProfiniteDensity:JustInfinite}
	Let $G$ be a finitely generated just-infinite group and $H \leq G$ be an infinite commensurated subgroup. Then the residual closure of  $H$ in $G$ is of finite index.
\end{cor}
\begin{proof}
Let $N := \bigcap_{g \in G} g \widetilde H g^{-1} \normal G$ be the normal subgroup afforded by the Main Theorem. 
Since $H$ is infinite and $N$ contains a finite index subgroup of $H$, we see that $N$ is infinite. The group $N$ is thus of finite index in the just-infinite group $G$. Hence, $\widetilde H$ is also of finite index.
\end{proof}

In \cite{Wesolek_Branch}, P. Wesolek shows that every commensurated subgroup of a finitely generated just-infinite branch group is either finite or of finite index. 
As an application of the Main  Theorem,  
we obtain the following related result. A \textbf{maximal subgroup} of a group is a subgroup that is maximal among all proper subgroups. 

\begin{cor}\label{cor:NoCommensurated}
Let $G$ be a finitely generated just-infinite group. Assume that whenever $M<H\le G$ are subgroups of $G$ such that $M$ is maximal in $H$ and $H$ has finite index in $G$, then $M$ has finite index in $G$. Then every commensurated subgroup of $G$ is either finite or has finite index in $G$.	  
\end{cor}

\begin{proof}
	Let $H \leq G$ be an infinite commensurated subgroup.  By Corollary~\ref{cor:ProfiniteDensity:JustInfinite}, the residual closure $\widetilde H$ is of finite index in $G$. Suppose toward a contradiction that $H<\widetilde{H}$. Let $M < \widetilde H $ be a maximal subgroup containing $H$, which exists since $\widetilde H$ is finitely generated. By hypothesis, the group $M$ is  of finite index in $\widetilde H$. In particular, $M$ is of finite index in $G$ and thus weakly separable. On the other hand, $\widetilde{H}$ is the smallest weakly separable subgroup of $G$ containing $H$. We infer that $M = \widetilde H$, which is absurd. Hence, $H=\widetilde{H}$ and so has finite index in $G$.
\end{proof}

Examples of groups satisfying the hypotheses of Corollary~\ref{cor:NoCommensurated} include the Grigorchuk group (see   \cite[Lemma~4]{GriWil}) and many related finitely generated torsion branch groups (see \cite{AlKlTh} and  \cite{Per2}), so we recover the corresponding special cases of Wesolek's result \cite{Wesolek_Branch}. 

Another striking family of just-finite groups is that consisting of the `residually finite Tarski monsters' constructed by M.~Ershov and A.~Jaikin in \cite{ErshovJaikin}. In addition to their residual finiteness, these groups enjoy the property that each of their finitely generated subgroups is finite or of finite index. In particular, these groups are \textbf{LERF}:  every finitely generated subgroup is separable.

The monster groups from \cite{ErshovJaikin} indeed enjoy a stronger property. For $p$ a prime, we say that a subgroup is \textbf{$p$-separable} if it is the intersection of subgroups of $p$ power index. We say that a group is \textbf{$p$-LERF} if every finitely generated subgroup is $p$-separable. The examples from \cite{ErshovJaikin} are virtually \textbf{$p$-LERF}.

 The following application, which applies to the just-infinite groups constructed in  \cite{ErshovJaikin}, was pointed out to us by A.~Jaikin. 

\begin{cor}\label{cor:p-LERF}
 Let $G$ be a finitely generated just-infinite group which is virtually $p$-LERF for some prime $p$. Then every commensurated subgroup of $G$ is either finite or of finite index.	  
\end{cor}

Since the Grigorchuk group (which is a $2$-group) is LERF by Theorem 2 of \cite{GriWil} and the Gupta--Sidki $3$-group is LERF by Theorem 2 of \cite{Garrido}, and upon noting that a $p$-group that is LERF is automatically $p$-LERF, we thus also recover two more special cases of Wesolek's result in \cite{Wesolek_Branch}. 

To prove Corollary \ref{cor:p-LERF} we use the following  subsidiary fact. 

\begin{lem}\label{lem:p-LERF}
	 Let $G$ be a finitely generated   group which is $p$-LERF for some prime $p$. If a  subgroup $H \leq G$ is of infinite index in $G$, then its  profinite  closure $\overline{H}$ is also of infinite index. 
\end{lem}
\begin{proof}
If $\overline{H}$ is of finite index in $G$, then the closure of the image of $H$ in the pro-$p$ completion  $G_{\widehat p}$ is of finite index. The closure in $G_{\widehat p}$ is then topologically finitely generated, hence  $H$ contains a finitely generated subgroup $H_0$  whose image in $G_{\widehat p}$ has the same closure as the image of  $H$. Since $H$ is of infinite index, so is  $H_0$, but the closure of the image of $H_0$   in $G_{\widehat p}$ is of finite index. This contradicts that  $G$ is $p$-LERF.
\end{proof}

\begin{proof}[Proof of Corollary~\ref{cor:p-LERF}]
Let $H \leq G$ be an infinite commensurated subgroup and let $G_0$ be a subgroup of finite index in $G$ which is $p$-LERF. By Corollary~\ref{cor:ProfiniteDensity:JustInfinite}, the residual closure $\widetilde H$ is of finite index, so the profinite closure is also of finite index. The profinite  closure of   $H_0 := H \cap G_0$ in  ${G_0}$ is thus of finite index in $G_0$.  By Lemma~\ref{lem:p-LERF}, this implies that the index of  $H_0$ in $G_0$ is finite. The index of  $H$ in $G$ is thus finite. 
\end{proof}

\begin{acknowledgement}
We would like to thank the Isaac Newton Institute for Mathematical Sciences, Cambridge for support and hospitality during the programme \textit{Non-positive curvature group actions and cohomology} where part of the work on this paper was accomplished. 
\end{acknowledgement}


\def\cprime{$'$}

\end{document}